\documentclass[11pt]{amsart}

\usepackage[T1]{fontenc}
\usepackage{amsmath,amssymb,mathtools}
\usepackage{microtype}
\usepackage[
  hidelinks,
  pdfauthor={Xinan Dai, Wenhao Deng, Yingdong Shi, Tailin Wu, and Yuchen Yang}
]{hyperref}

\allowdisplaybreaks

\newtheorem{theorem}{Theorem}[section]
\newtheorem{lemma}[theorem]{Lemma}
\newtheorem{corollary}[theorem]{Corollary}
\theoremstyle{remark}
\newtheorem{remark}[theorem]{Remark}

\newcommand{\Q}{\mathbb{Q}}
\newcommand{\Z}{\mathbb{Z}}
\newcommand{\cminus}{C^{-}}

\title[Stanley's Problem 4 on differential posets]
{A Negative Answer to Stanley's Problem 4 on Differential Posets}

\author{Xinan Dai}
\address{Key Laboratory for Information Science of Electromagnetic Waves, College of Future Information Technology, Fudan University, Shanghai, China}
\curraddr{Department of Artificial Intelligence, School of Engineering, Westlake University, Hangzhou, China}
\email{xndai23@m.fudan.edu.cn}

\author{Wenhao Deng}
\thanks{Wenhao Deng is a student at the University of Glasgow and is currently an intern at the AI for Scientific Simulation and Discovery Lab, Westlake University.}
\address{University of Glasgow, Glasgow, United Kingdom}
\curraddr{Department of Artificial Intelligence, School of Engineering, Westlake University, Hangzhou, China}
\email{dengwenhao@westlake.edu.cn}

\author{Yingdong Shi}
\address{School of Information Science and Technology, ShanghaiTech University, Shanghai, China}
\email{shiyd2023@shanghaitech.edu.cn}

\author{Tailin Wu}
\address{Department of Artificial Intelligence, School of Engineering, Westlake University, Hangzhou, China}
\email{wutailin@westlake.edu.cn}

\author{Yuchen Yang}
\address{Department of Artificial Intelligence, School of Engineering, Westlake University, Hangzhou, China}
\email{yangyuchen@westlake.edu.cn}

\date{27 July 2026}
\subjclass[2020]{Primary 06A07; Secondary 05A15, 05E10}
\keywords{Differential poset, multichain generating function, rational formal power series, reflection construction, crown reflection}

\begin{document}

\begin{abstract}
In Problem~4 of his 1988 paper on differential posets, Stanley asked whether the weighted $k$-multichain series of an arbitrary differential poset is always a rational multiple of the $k$th power of its rank series. We answer this question in the negative. Already for $k=2$, we construct a locally finite $1$-differential poset $P$ for which
\[
\frac{M_{P,2}(q)}{F_P(q)^2}
\]
is not rational over any field of characteristic zero. The construction uses Wagner's crown reflection. At each sufficiently high rank there are two valid one-rank extensions with the same required Gram matrix and with new rank sizes differing by one. We show that the first quotient coefficient affected by that rank distinguishes the two choices. Choosing successively between them and diagonalizing against the rational power series produces the required poset. The same coefficient argument also yields continuum many distinct quotient series, of which continuum many are nonrational.
\end{abstract}

\maketitle
\tableofcontents

\section{Introduction}

\subsection{Differential posets}
Let
\[
P=\bigsqcup_{n\geq 0}P_n
\]
be a locally finite graded poset with least element $\widehat{0}$. For a positive integer $r$, Stanley calls $P$ \emph{$r$-differential} if the following conditions hold \cite[Definition~1.1]{Sta88}:
\begin{enumerate}
\item[(D1)] $P$ is locally finite and graded, with least element $\widehat{0}$;
\item[(D2)] two distinct elements have as many common upper covers as common lower covers;
\item[(D3)] an element with $d$ lower covers has $d+r$ upper covers.
\end{enumerate}
Each rank $P_n$ is finite. Put $p_n=|P_n|$, and let
\[
A_n\in\{0,1\}^{p_{n-1}\times p_n},
\qquad
(A_n)_{xy}=1\quad\Longleftrightarrow\quad x\lessdot y,
\]
be the incidence matrix between ranks $P_{n-1}$ and $P_n$. If $U$ and $D$ are the up and down operators, then Stanley's relation $DU-UD=rI$ is equivalent to
\begin{equation}
A_nA_n^{\mathsf T}-A_{n-1}^{\mathsf T}A_{n-1}=rI_{p_{n-1}},
\label{eq:gram}
\end{equation}
where $A_0$ is the empty $0\times1$ matrix \cite[Theorem~2.2]{Sta88}.

\subsection{Stanley's Problem 4}
The rank generating series of $P$ is
\[
F_P(q)=\sum_{n\geq0}p_nq^n.
\]
For a fixed positive integer $k$, define the weighted $k$-multichain series by
\begin{equation}
M_{P,k}(q)=
\sum_{\widehat{0}\leq x_1\leq\cdots\leq x_k}
q^{\rho(x_1)+\cdots+\rho(x_k)}.
\label{eq:multichain}
\end{equation}
Thus the least element is fixed below the $k$ displayed entries. This is the convention compatible with Stanley's phrase ``$k$-element multichains'' and with the identity for $k=1$ in the discussion following equation~(72) of \cite{Sta88}.

For Young's lattice $Y$, Stanley obtained
\[
M_{Y,k}(q)=\prod_{i\geq1}(1-q^i)^{-\min(i,k)}.
\]
Since $F_Y(q)=\prod_{i\geq1}(1-q^i)^{-1}$, it follows that
\begin{equation}
\frac{M_{Y,k}(q)}{F_Y(q)^k}
=
\prod_{1\leq i<k}(1-q^i)^{k-i},
\label{eq:young}
\end{equation}
a polynomial in $q$. Stanley then asked in Problem~4 whether an analogous rational-factor formula holds for every differential poset: namely, whether for each $P$ and fixed $k$ one can write
\begin{equation}
M_{P,k}(q)=R(q)F_P(q)^k
\label{eq:p4}
\end{equation}
with $R(q)$ rational. The case $k=1$ says nothing beyond the definition. Our example shows that \eqref{eq:p4} fails for $k=2$.

\subsection{Main theorem}
\begin{theorem}\label{thm:main}
There is a recursively specified locally finite $1$-differential poset $P$ such that
\[
Q_P(q):=\frac{M_{P,2}(q)}{F_P(q)^2}
\]
is not rational over any field of characteristic zero. Consequently, Stanley's Problem~4 has a negative answer.
\end{theorem}

The construction starts from Stanley's reflection extension and modifies one small part of it by a crown reflection. The ordinary and modified extensions have identical row Gram matrices, so both satisfy the differential relation, but their new ranks differ in size by one. That difference appears in a single, newly determined coefficient of $Q_P(q)$. Repeating the choice at successive ranks turns the family into a diagonal argument against all rational series.

\section{Two extensions with the same Gram matrix}

A finite graded poset through rank $n$ will be called \emph{$1$-differential through rank $n$} if (D2) and (D3) hold below its top rank. Equivalently, its incidence matrices satisfy \eqref{eq:gram} with $r=1$ for $1\leq i\leq n$.

Stanley's reflection construction extends any such finite prefix by one rank \cite[Proposition~6.1]{Sta88}. On the point set $P_n$, take the blocks
\begin{equation}
B_y=\{x\in P_n:y\lessdot x\}
\qquad (y\in P_{n-1})
\label{eq:blocks}
\end{equation}
and, in addition, one singleton block $\{x\}$ for every $x\in P_n$. These blocks, counted with multiplicity, are the lower-cover sets of the elements in the new rank.

\begin{lemma}[Crown trade]\label{lem:crown}
Suppose a finite $1$-differential prefix through rank $n$ contains an element $y\in P_{n-1}$ with exactly two lower covers. Then the prefix has two extensions through rank $n+1$ whose new rank sizes differ by one.
\end{lemma}

\begin{proof}
Condition (D3) gives $|B_y|=3$; write $B_y=\{a,b,c\}$. In the ordinary reflection extension, the four blocks involving this part of the construction are
\[
\{a,b,c\},\qquad \{a\},\qquad \{b\},\qquad \{c\}.
\]
Replace them by
\begin{equation}
\{a,b\},\qquad \{a,c\},\qquad \{b,c\}.
\label{eq:trade}
\end{equation}
On either side, each of $a,b,c$ occurs twice, and each pair of distinct points occurs together once. No other incidence changes. Hence the row Gram matrix $A_{n+1}A_{n+1}^{\mathsf T}$ is unchanged, and \eqref{eq:gram} remains valid. The ordinary extension uses four blocks here, whereas \eqref{eq:trade} uses three, so the latter has one fewer element in rank $n+1$.

This local replacement is Wagner's crown reflection; see Lewis \cite[Theorem~3.4 and pp.~11--12]{Lew07}.
\end{proof}

The next lemma ensures that the same choice can be made again at every later stage.

\begin{lemma}[A persistent crownable element]\label{lem:persistent}
Every finite $1$-differential prefix with top rank at least four contains an element in the penultimate rank having exactly two lower covers and three upper covers. After either extension in Lemma~\ref{lem:crown}, the resulting prefix has the same property.
\end{lemma}

\begin{proof}
The ranks through $P_2$ are forced. Let $a_0=\widehat{0}$, let $a_1$ be the unique element of rank one, and denote the two elements of rank two by $a_2$ and $b_2$. We claim that every subsequent valid extension contains elements $a_i,b_i$ for which
\begin{equation}
\cminus(a_{i+1})=\{a_i\},
\qquad
\cminus(b_{i+1})=\{a_i,b_i\}
\qquad (i\geq2).
\label{eq:threads}
\end{equation}
Assume these elements have been chosen through rank $i$. The element $a_{i-1}$ has one lower cover, and therefore exactly two upper covers, namely $a_i$ and $b_i$. Since $a_i$ and $b_i$ share the lower cover $a_{i-1}$, condition (D2) gives a unique common upper cover $x$ in rank $i+1$. The element $a_i$ has one lower cover, so it has one further upper cover $z$.

We show that no additional element lies below either $x$ or $z$. If $w$ were such an element, then $a_i$ and $w$ would share an upper cover. By (D2), they would also share a lower cover. The only possibility is $a_{i-1}$, whose upper covers are $a_i$ and $b_i$. Thus $w=b_i$. This is already one of the lower covers of $x$, while for $z$ it would give a second common upper cover of $a_i$ and $b_i$. Both possibilities contradict the choice of $w$ as an additional lower cover. Hence
\[
\cminus(x)=\{a_i,b_i\},
\qquad
\cminus(z)=\{a_i\}.
\]
Taking $b_{i+1}=x$ and $a_{i+1}=z$ proves \eqref{eq:threads} inductively.

Now let the top rank be $P_n$ with $n\geq4$. The element $b_{n-1}$ has exactly two lower covers and, by (D3), three upper covers in $P_n$. It therefore supports the crown trade. Since the preceding argument applies to every valid one-rank extension, the required configuration is present again after either choice.
\end{proof}

Start with the ordinary reflection prefix through rank four, whose rank sizes are
\[
1,1,2,3,5.
\]
Fix labels $a_2,b_2$ for the two elements of rank two, and thereafter continue the distinguished chains $(a_i)$ and $(b_i)$ by the unique choices made in the proof of Lemma~\ref{lem:persistent}. Thus, when a prefix has been fixed through rank $N-1$, the element $b_{N-2}$ is a specified crownable element: it has exactly two lower covers and three upper covers in rank $N-1$. At rank $N$ we apply either the ordinary reflection extension or the crown trade at this distinguished element.

By Lemmas~\ref{lem:crown} and \ref{lem:persistent}, this binary choice is available at every rank $N\geq5$. An infinite sequence of choices has a compatible union, and that union is a locally finite $1$-differential poset: each instance of (D2) and (D3) is decided inside some finite prefix. Once the initial labels and the binary sequence are fixed, the construction is recursive. Lewis used this family to obtain uncountably many nonisomorphic differential posets \cite{Lew07}. Here we use the same family for a different purpose, by reading its choices from the coefficients of a generating-function quotient.

\section{The quotient detects the first differing rank}

Since both $M_{P,2}(q)$ and $F_P(q)^2$ have integer coefficients and constant term one, the quotient
\[
Q_P(q)=\frac{M_{P,2}(q)}{F_P(q)^2}
\]
is a well-defined element of $\Z[[q]]$.

\begin{lemma}[First differing coefficient]\label{lem:coefficient}
Let two graded prefixes agree through rank $N-1$ and differ only in their rank-$N$ extensions. If
\[
\delta=p_N'-p_N,
\]
then
\[
[q^i]Q_{P'}=[q^i]Q_P\quad (i<N)
\]
and
\begin{equation}
[q^N]Q_{P'}-[q^N]Q_P=-\delta.
\label{eq:coefficient-change}
\end{equation}
\end{lemma}

\begin{proof}
Nothing below degree $N$ depends on rank $N$, so the quotient coefficients below that degree agree. A rank-$N$ element contributes to a two-multichain of total weight $N$ only in the pair $(\widehat{0},x)$. Therefore
\begin{equation}
[q^N]\bigl(M_{P',2}-M_{P,2}\bigr)=\delta.
\label{eq:M-change}
\end{equation}
In $F_P(q)^2$, a new rank-$N$ element can be paired with the rank-zero term in either order, and no other new product has degree $N$. Hence
\begin{equation}
[q^N]\bigl(F_{P'}^2-F_P^2\bigr)=2\delta.
\label{eq:F-change}
\end{equation}
Use $M_{P,2}=Q_PF_P^2$ and compare degree $N$. The two quotients and the two squared rank series have the same lower coefficients, and all have constant term one where appropriate. Combining \eqref{eq:M-change} and \eqref{eq:F-change} gives
\[
[q^N](Q_{P'}-Q_P)=\delta-2\delta=-\delta,
\]
which is \eqref{eq:coefficient-change}.
\end{proof}

At rank $N$, ordinary reflection has one more element than crown reflection. Thus, if ordinary reflection gives the coefficient $c=[q^N]Q_P$, crown reflection gives $c+1$. Once rank $N$ has been fixed, later extensions cannot change this coefficient.

\section{Diagonalization against rational series}

We first note that enlarging the coefficient field does not affect rationality for a series with rational coefficients.

\begin{lemma}\label{lem:field}
Let $K$ be a field of characteristic zero and let $S(q)\in\Q[[q]]$. If $S(q)$ is rational over $K$, then it is rational over $\Q$.
\end{lemma}

\begin{proof}
Embed $\Q$ into $K$. If $S$ is rational over $K$, its coefficients satisfy a linear recurrence of some finite order:
\begin{equation}
s_n+b_1s_{n-1}+\cdots+b_ds_{n-d}=0
\qquad (n\geq n_0),
\label{eq:recurrence}
\end{equation}
with $b_1,\ldots,b_d\in K$. These are linear equations in the finitely many unknowns $b_1,\ldots,b_d$, and all their coefficients lie in $\Q$. Regard each equation as an augmented row vector in the finite-dimensional space $\Q^{d+1}$. Finitely many of these row vectors span all the others. The corresponding finite rational linear system is consistent over $K$ and hence, by Gaussian elimination over $\Q$, has a solution in $\Q^d$. Every remaining equation is a rational linear combination of the chosen ones, so this rational solution satisfies the full recurrence \eqref{eq:recurrence}. Therefore $S(q)$ is rational over $\Q$.
\end{proof}

\begin{theorem}\label{thm:construction}
There exists a recursively specified locally finite $1$-differential poset $P$ such that $Q_P(q)$ is not rational over any field of characteristic zero.
\end{theorem}

\begin{proof}
List all elements of $\Q(q)$ that are regular at $q=0$ and have constant term one as
\[
R_1(q),R_2(q),\ldots.
\]
Such a list may be obtained, for example, by enumerating pairs of integer polynomials representing their numerators and denominators. Repetitions in the list are harmless.

Begin with the ordinary $1$-differential prefix through rank four, with the distinguished labels and chains fixed as above. At stage $m$, set $N=m+4$, so the current prefix is already fixed through rank $N-1$. Form its ordinary reflection extension through rank $N$ and use the distinguished crownable element $b_{N-2}$ for the alternative crown extension. Compute the coefficient
\[
c=[q^N]Q
\]
from that finite prefix. If $c\neq[q^N]R_m(q)$, keep the ordinary extension. If equality holds, take the crown extension instead. Lemma~\ref{lem:coefficient} changes the coefficient in the second case from $c$ to $c+1$.

The construction can be continued indefinitely. Its union is a locally finite $1$-differential poset, and for every $m$ the resulting quotient differs from $R_m$ in degree $m+4$. Hence
\[
Q_P(q)\notin\Q(q).
\]
Lemma~\ref{lem:field} rules out rationality over every characteristic-zero field.
\end{proof}

Theorem~\ref{thm:main} follows immediately. The same argument records more than the existence of one counterexample.

\begin{corollary}\label{cor:continuum}
The crown/reflection family gives continuum many distinct quotient series $Q_P(q)$. Continuum many of them are nonrational over every field of characteristic zero.
\end{corollary}

\begin{proof}
Take two different infinite binary choice sequences, and let $N$ be the first rank at which they differ. Lemma~\ref{lem:coefficient} shows that the corresponding quotient series agree below degree $N$ and differ in degree $N$. The map from choice sequences to quotient series is therefore injective, so its image has cardinality continuum.

There are only countably many rational functions in $\Q(q)$. By Lemma~\ref{lem:field}, a series in $\Q[[q]]$ that is rational over a characteristic-zero field is already rational over $\Q$. It follows that all but countably many series in the image are nonrational over every such field.
\end{proof}

\begin{remark}
The proof is entirely theoretical. A finite computation can illustrate the crown trade or verify initial coefficients, but no search or computer-assisted step enters Theorem~\ref{thm:construction}.
\end{remark}

\section{Relation to earlier work}

The reflection/crown family is due to earlier work. Wagner introduced the crown operation, and Lewis used repeated reflection and crown choices to construct uncountably many differential posets \cite{Lew07}. The point needed here is different: Lemma~\ref{lem:coefficient} shows that the first rank at which two choices differ is visible in the quotient $M_{P,2}(q)/F_P(q)^2$. This observation turns the family into both an injective encoding by formal power series and a diagonal construction against rational functions.

Some related generating-function formulas concern statistics other than \eqref{eq:multichain}. Butler studied chains of ordinary partitions \cite{But89}. Stanley's later formulas for multichains in the Fibonacci lattices $\mathrm{Fib}(r)$ and $Z(r)$ weight the rank of the top element, rather than the sum of the ranks of all entries \cite{Sta90}. The Young--Fibonacci series considered by Petrov and Scott likewise arise from a different weighting \cite{PS26}. These results therefore do not imply the rationality asked for in Problem~4.

A search of the differential-poset literature available to the authors did not reveal an earlier negative answer to Problem~4 or the coefficient-separation argument above. As usual, a literature search cannot by itself establish priority. The claim of the present paper is the precise theorem proved here: the rational-factor formula proposed in Problem~4 fails, already for a $1$-differential poset and $k=2$.

\section*{Statement on AI-assisted discovery and human verification}

The counterexample were found by the TARS agent system during an autonomous mathematical search. Xinan Dai completed proof reconstruction, literature review and manuscript writing and verified all mathematical results manually.

\end{document}